\documentclass[12pt]{amsart}
\usepackage[english]{babel}
\usepackage{amsfonts}
\usepackage{amsthm}
\usepackage{amsbsy}
\usepackage{amssymb}
\usepackage{graphicx}
\usepackage{eso-pic}
\usepackage{tikz}
\usepackage{xfrac}
\usepackage{float}
\usepackage[bookmarks=true,colorlinks=true,urlcolor=black,linkcolor=black,citecolor=black,pagebackref=false]{hyperref}
\usepackage{bookmark}
\usepackage{amsmath}
\usepackage{subfigure}
\usepackage{hyperref}
\usepackage{scalerel,stackengine}
\usepackage{epigraph}
\usepackage{placeins}
\usetikzlibrary{matrix,arrows}
\usepackage[letterpaper,margin=3.5cm,top=1in,bottom=1in]{geometry}
\usepackage{todonotes}

\makeatletter
\newcommand\footnoteref[1]{\protected@xdef\@thefnmark{\ref{#1}}\@footnotemark}
\makeatother

\newtheorem{lemma}{Lemma}[section]
\newtheorem{thm}[lemma]{Theorem}
\newtheorem{prop}[lemma]{Proposition}
\newtheorem{cor}[lemma]{Corollary}
\newtheorem*{cor*}{Corollary}

\theoremstyle{definition}

\newtheorem{quest}[lemma]{Question}

\theoremstyle{remark}

\newcommand{\matR} {\ensuremath {\mathbb{R}}}

\newcommand{\matZ} {\ensuremath {\mathbb{Z}}}

\newcommand{\SO}{\ensuremath {\mathrm{SO}}}

\newcommand{\Isom}{\ensuremath {\mathrm{Isom}}}

\stackMath
\newcommand\reallywidehat[1]{%
\savestack{\tmpbox}{\stretchto{%
  \scaleto{%
    \scalerel*[\widthof{\ensuremath{#1}}]{\kern-.6pt\bigwedge\kern-.6pt}%
    {\rule[-\textheight/2]{1ex}{\textheight}}
  }{\textheight}%
}{0.5ex}}%
\stackon[1pt]{#1}{\tmpbox}%
}
\author{Alexander Kolpakov}
\address{Institut de Math\'ematiques, Universit\'e de Neuch\^atel, Rue Emile-Argand 11, Neuch\^atel, Suisse / Switzerland}
\email{kolpakov dot alexander at gmail dot com}

\author{Stefano Riolo}
\address{Section de Math\'ematiques, Universit\'e de Gen\`eve, Rue du Conseil-G\'en\'eral 7-9, 1205 Gen\`eve, Suisse / Switzerland}
\email{stefano dot riolo at unige dot ch}

\author{Steven T. Tschantz}
\address{Department of Mathematics, Vanderbilt University, 1326 Stevenson Center Ln, Nashville, TN 37212, U.S.A.}
\email{steven dot tschantz at vanderbilt dot edu}

\thanks{A. K. and S. R. were supported by the SNSF project no. PP00P2-170560. S. R. has been supported by the  MIUR--PRIN project 2017JZ2SW5 and the SNSF ``Ambizione'' grant PZ00P2-193559. }

\title[The signature of cusped hyperbolic $4$-manifolds]{The signature of cusped  hyperbolic $4$-manifolds}

\begin{document}

\begin{abstract}
In this note we show that every integer is the signature of a non-compact, oriented, hyperbolic $4$--manifold of finite volume, and give some partial results on the geography of such manifolds. The main ingredients are a theorem of Long and Reid, and the explicit construction of a hyperbolic $24$--cell manifold with some special topological properties.
\end{abstract}

\maketitle

\epigraph{\textit{Few things are harder to put up with than the annoyance of a good example.}}{-- Mark Twain}

\section{Introduction}

In the recent survey \cite{M} on hyperbolic $4$--manifolds, Martelli asks whether
one can find a cusped hyperbolic $4$--manifold with non--vanishing signature (see \cite[Section 4]{M} for other open questions).
The main purpose of this note is to prove the following:

\begin{thm} \label{thm:signature_spectrum}
Every integer is the signature of a cusped hyperbolic $4$--man\-i\-fold.
\end{thm}

All manifolds in the paper are assumed connected and oriented unless otherwise stated. Hyperbolic manifolds are understood to be complete and of finite volume. Non--compact hyperbolic manifolds are called \emph{cusped}. 

Closed hyperbolic $4$--manifolds have vanishing signature, while for cusped manifolds this property holds virtually (in a strong sense, see Corollary \ref{cor:virtually_0}). The latter fact follows from a result of Long and Reid \cite{LR} which plays an important role in this paper. It is worth noting that for cusped manifolds the signature is not necessarily multiplicative under finite coverings.

As a byproduct of our construction, we obtain some results on the ``geography problem'' for cusped hyperbolic $4$--manifolds. The latter asks about realising a given pair of of integers as the Euler characteristic $\chi(M)$ and signature $\sigma(M)$ of a cusped hyperbolic $4$--manifold $M$.
Indeed, Theorem \ref{thm:signature_spectrum} is a consequence of the following:

\begin{thm}\label{thm:main}
For every pair of positive integers $m, n$ with $m$ odd, there exists a cusped hyperbolic $4$--manifold $M$ with $\sigma(M) = \pm n$ and $\chi(M) = mn$.
\end{thm}

Recall that hyperbolic $4$--manifolds have positive Euler characteristic by the generalised Gau\ss--Bonnet formula.

\subsection{Consequences and questions}

As shown by Ratcliffe and Tschantz \cite{RT}, every positive integer is realised as $\chi(M)$ for some cusped hyperbolic $4$--manifold $M$. However, all the manifolds constructed in \cite{RT} (as well as their covers) happen to have vanishing signature, as well as any other cusped hyperbolic $4$--manifold that we could find in the literature. Combining this fact (or Corollary \ref{cor:virtually_0}) with Theorem \ref{thm:main}, we obtain that for every integer $n$ there exists a cusped hyperbolic $4$--manifold $M$ with $\sigma(M) = n$ and $\chi(M)$ arbitrarily big.

On the other hand (Proposition \ref{prop:universal-constant}), every cusped hyperbolic $4$--manifold $M$ satisfies
$$\chi(M) > 0.03493 \cdot |\sigma(M)|.$$
It thus seems reasonable trying to minimise $\chi$ for any fixed value of $|\sigma|$. For $\sigma = 0$, the minimum possible $\chi = 1$ is realised by \cite{RT}. Theorem \ref{thm:main} implies the following fact.

\begin{cor} \label{cor:chi=sigma}
For every positive integer $n$, there exists a cusped hyperbolic $4$--manifold with $\chi(M) = \sigma(M) = n$.
\end{cor}

We shall call the quantity
$$\alpha(M) = \sigma(M)/\chi(M)$$
\textit{the slope} of $M$. By the above inequality, the slope of a cusped hyperbolic $4$--manifold is always bounded, and a natural ``geography'' problem is to determine which slopes can be realised. Theorem \ref{thm:main} gives a partial answer:

\begin{cor}\label{cor:slopes} 
For every odd integer $m$, there exist infinitely many cusped hyperbolic $4$--manifolds with slope $1/m$.
\end{cor}

In particular, the maximum slope realised by our construction equals $1$. Note that a well--known conjecture dating back to Gromov's work on bounded cohomology \cite[\S 8.A4]{Gromov} states that every \emph{closed} aspherical $4$--manifold $M$ satisfies $|\alpha(M)| \leq 1$, which is known as Winkelnkemper's inequality \cite{JK}. 

In the setting of our work, the following questions arise naturally. 

\begin{quest}
Can we describe the set of pairs $(\chi(M), \sigma(M))$ for all possible cusped hyperbolic $4$--manifolds $M$? In other words, what is the geography of cusped hyperbolic $4$--manifolds?
\end{quest}

\begin{quest}
What is the maximum slope of a cusped hyperbolic $4$--manifold?
\end{quest}

\subsection{On the proof}

Our proof of Theorem \ref{thm:main} is constructive and rather simple. We explicitly build a cusped hyperbolic $4$--manifold $M$ satisfying $\chi(M) = \sigma(M) = 1$ and such that for every $m, n \geq 1$ with $m$ odd, there exists an $mn$--sheeted covering $M_{m,n} \to M$ with $\sigma(M_{m,n}) = n$. 

An essential tool is an adaptation of the Atiyah--Patodi--Singer formula for cusped hyperbolic $4$--manifolds by Long and Reid (cf. Theorem \ref{thm:long-reid}), combined with the results of Ouyang \cite{O}, thanks to which the signature can be expressed only in terms of the oriented homeomorphism classes of the cusp sections.

Similar to \cite{RT, RT2}, the manifold $M$ is obtained by gluing the sides of the ideal right-angled $24$--cell. In particular, $\chi(M) = 1$. The side pairing is performed in order to have one cusp with section a ``quarter twist'' flat $3$--manifold $F_4$, while the remaining cusps have $3$--torus sections. Then Long and Reid's signature formula gives  $\sigma(M) = \pm1$.

Moreover, there exist 
homomorphisms $\pi_1(M) \to \matZ$ that map to zero some generators of the parabolic subgroup of the $F_4$--cusp, in a way to get an $n$--sheeted cyclic covering $M_{n} \to M$ under which $\sigma$ is multiplicative, and an $m$--sheeted cyclic covering $M_{m,n} \to M_n$ under which $\sigma$ does not change. The latter condition is satisfied if $m$ is odd: in this case, the manifold $M_{m,n}$ has exactly $n$ cusps of type $F_4$, all coherently oriented, while the remaining ones have $3$--torus sections. This implies $\sigma(M_{m,n}) = \pm n$.

As in \cite{RT2}, the starting point to find $M$ is an extensive computer search. Despite the fact that $M$ is produced by computer, it has a relatively simple structure and all the aforementioned properties can be verify by hand. However, finding such a starting manifold $M$ by hand is conceivably impossible within any reasonable time if one wants to search through all or a sufficiently large number of side pairings. 

We would like to stress the fact that the manifold $M$ used in our present construction does not appear to be special: we only concentrate on it because of its simple structure that makes our proofs verifiable by hand. Another construction of Tschantz's, the details of which are rather technical and impossible to verify without computer aid, gives on the order of $m^m$ commensurability classes of manifolds with $\chi \leq m$ for any fixed value of $\sigma$ and $m$ large enough (cf. \cite{BGLM, GL}). Presenting the details here, however, would obfuscate our main goal, that is proving Theorem \ref{thm:signature_spectrum} in a relatively simple way.

\subsection{Structure of the paper}

Some general facts on the geography of cusped hyperbolic $4$--manifolds are given in Section \ref{sec:preliminaries}. The proof of Theorem \ref{thm:main} follows in Section \ref{sec:proof}.

\subsection*{Acknowledgements}
The authors
are grateful to Bruno~Martelli for suggesting a strategy that helped simplifying
their
previous proof of Theorem \ref{thm:signature_spectrum}. They would also like to thank John~Ratcliffe and Alan~Reid for showing interest in this work.

\section{Preliminaries} \label{sec:preliminaries}

Recall that all manifolds in this paper are connected and oriented, unless otherwise stated. All homology and cohomology groups are understood with integer coefficients.

\subsection{Signature and geography}

Let $X$ be a compact $4$--manifold with boundary, and let $[X, \partial X] \in H_4(X, \partial X) \cong \mathbb{Z}$ be its fundamental class.

The cup product on $H^2(X, \partial X)$ defines the following symmetric bilinear form, called the \emph{intersection form} of $X$:
\begin{equation*}
H^2(X, \partial X) \times H^2(X, \partial X) \to \mathbb{Z}, \quad (\alpha, \beta) \mapsto \alpha \smile \beta \, ([X, \partial X]).
\end{equation*}
By the Poincar\'e--Lefschetz duality, the radical of the intersection form is the kernel of the natural map $H^2(X, \partial X) \to H^2(X)$.

Let $q_+$ and $q_-$ be the positive and negative inertia indices of the associated quadratic form over $\matR$. Then the \emph{signature} of $X$ is defined as
$$\sigma(X) = q_+ - q_{-} \in \matZ.$$

The notion of ``geography'' for $4$--manifolds appears to be classical and has been studied in different contexts by many authors (see \cite{Stipsicz} for a detailed survey). For $M$ a $4$--manifold, the \textit{geography map} can be defined as $M \mapsto (\chi(M),\, \sigma(M))$. The \textit{slope} of $M$ is $\alpha(M) = \sigma(M)/\chi(M)$. Note that in \cite{Stipsicz} the slope and geography map are defined for manifolds with complex structures and thus expressed via Chern numbers: we slightly modify the definitions in our setting.

Here and below we shall be interested in the class of cusped hyperbolic $4$--manifolds, and the behaviour of the geography and slope maps on it. 

\subsection{Signature of hyperbolic $4$--manifolds}

If $M$ is a closed hyperbolic $4$--manifold, then it follows from the Hirzebruch signature theorem theorem that $\sigma(M) = 0$, since by a theorem of Chern the first Pontryagin class vanishes \cite[Theorem 11.3.3]{R} (as it more generally does for locally conformally flat manifolds, cf. \cite{LR}). 

Let now $M$ be a cusped hyperbolic $4$--manifold. Since $M$ is homeomorphic to the interior of a compact $4$--manifold $X$ with boundary, then by the Poincar\'e--Lefschetz duality we have a well-defined intersection form on $H_2(M)
\cong H^2(X, \partial X)$ and signature $\sigma(M) := \sigma(X)$.

In \cite{LR} Long and Reid provided an adaptation of the Atiyah--Patodi--Singer formula \cite{APS} for cusped hyperbolic $4$--manifolds. 

\begin{thm}[Long--Reid] \label{thm:long-reid}
Let $M$ be a hyperbolic $4$--manifold with $m$ cusps $C_1, \ldots, C_m$, and let $S_i$ be a horospherical section of $C_i$. Then
\begin{equation*}
\sigma(M) = - \sum_{i=1}^m \eta(S_i).
\end{equation*}
\end{thm}

Here $\eta$ is the so-called \emph{eta invariant} of a closed oriented Riemannian $3$--mani\-fold, see \cite{APS, LR}. Same as the signature, 
$\eta$ changes its sign when the orientation of the manifold is reversed, and thus vanishes on achiral manifolds. Since cusped hyperbolic manifolds have virtually torus cusps \cite[Theorem 3.1]{MRS}, we have the following corollary. 

\begin{cor} \label{cor:virtually_0}
For every cusped hyperbolic $4$--manifold $M$ there is a finite covering $M' \to M$ such that $\sigma(M'') = 0$ for every finite covering $M'' \to M'$.
\end{cor}

This shows that, in particular, constructing cusped hyperbolic manifolds with given signatures likely cannot be done by considering some sort of ``generic'' or ``random'' coverings of a given particular manifold. We shall concentrate on using cyclic coverings as in this case preserving the topological type of cusps in the covering is relatively easy.

There are precisely six closed orientable flat $3$--manifolds up to homeomorphism \cite{HW}. We denote them in the order given by Hantzsche and Wendt \cite{HW} as follows: the  $3$--torus $F_1$, the ``half-twist'' manifold $F_2$, the ``third-twist'' manifold $F_3$,  the ``quarter-twist'' manifold $F_4$, the ``sixth-twist'' manifold $F_5$, and the Hantzsche--Wendt manifold $F_6$. For $i = 1, \ldots, 5$, the manifold $F_i$ is a mapping torus over $S^1 \times S^1$ with monodromy of order $1$, $2$, $3$, $4$, and $6$, respectively.

As shown by Ouyang \cite{O}, the $\eta$--invariant of a flat $3$--manifold does not depend on the chosen flat metric, and thus represents a topological invariant for such manifolds. More precisely \cite{O}, we have:
$$\eta(F_1) = \eta(F_2) = \eta(F_6) = 0,$$
$$\eta(F_3) = \pm \frac{2}{3},\quad \eta(F_5) = \pm \frac{4}{3},\quad \eta(F_4) = \pm1.$$

In particular, a hyperbolic $4$--manifold $M$ whose unique non--torus cusp has $F_4$ section has signature $\pm 1$. Moreover, $\sigma(\widetilde{M}) = \pm n$ for any $n$--sheeted regular covering $\widetilde{M} \to M = \widetilde{M} \slash G$ with $G < \Isom^+(\widetilde{M})$ acting transitively on the $F_4$--cusps of $\widetilde{M}$.

\subsection{The slope is bounded}

The previous facts together with the volume estimates of Kellerhals \cite{K} imply the following lower bound, that we believe however to be far from sharp.

\begin{prop}\label{prop:universal-constant}
Every cusped hyperbolic $4$--manifold $M$ satisfies
$$|\alpha(M)| < 28.62869.$$
\end{prop}

\begin{proof}
From Formula (4.6) in \cite[Example 2]{K} we have $\mathrm{Vol}(M) > 0.61293 \cdot k$, where $k$ is the number of cusps of $M$. By Theorem \ref{thm:long-reid} and the values of $\eta(F_i)$ listed above, we have $\frac{4}{3} \cdot k \geq |\sigma(M)|$. Then the claimed inequality follows by applying the Gau\ss--Bonnet formula $\mathrm{Vol}(M) = \frac{4 \pi^2}{3} \cdot \chi(M)$. 
\end{proof}

\section{Proofs} \label{sec:proof}

In the following subsection we prove Theorem \ref{thm:main}. The essential geometric construction used in our proof (Theorem \ref{thm:construction}) is postponed to a separate subsection.

\subsection{The proof}

A \emph{hyperbolic $24$--cell manifold} is a hyperbolic $4$--manifold which can be obtained by gluing isometrically in pairs the sides of an ideal right-angled $24$--cell. Such a manifold $M$ satisfies $\chi(M) = 1$.

Recall that the fundamental group of a flat quarter-twist $3$--manifold $F_4$ is generated by two translations $t_1$ and $t_2$, and a rototranslation $a$ whose rotational part has order $4$. Note that $t_1$, $t_2$ and $t_3 = a^4$ generate the translation lattice of $F_4$. 

Given a hyperbolic $4$--manifold $M$ with a cusp $C$ of type $F_4$, let $\pi_1(C) = \langle t_1, t_2, a \rangle < \pi_1(M)$ denote the corresponding parabolic subgroup.

In the next subsection we shall prove the following theorem, that is the cornerstone of our construction. 

\begin{thm} \label{thm:construction}
There exist an orientable hyperbolic $24$--cell manifold $M$ with one cusp $C$ of type $F_4$ and all the other cusps of type $F_1$, and two surjective homomorphisms $h, v \colon \pi_1(M) \to \matZ$ such that $\pi_1(C) = \langle t_1, t_2, a \rangle \subset \ker(h)$, while $v(t_1) = v(t_2) = 0$ and $v(a) = 1$.
\end{thm}

We are ready to prove Theorem \ref{thm:main} assuming Theorem \ref{thm:construction}.

Since $\eta(F_4) = \pm 1$ and $\eta(F_1) = 0$, by Theorem \ref{thm:long-reid} $\sigma(M) = 1$ up to reversing the orientation on $M$. 

For $n \geq 1$, let $p_n \colon \matZ \to \matZ / n \matZ$ be the reduction mod $n$, and $M_n \to M$ be the cyclic $n$--sheeted covering associated to $h_n := p_n \circ h$. Since $\pi_1(C) \subset \ker(h_n)$, the manifold $M_n$ has $n$ cusps of type $F_4$. Moreover, there is an orientation-preserving isometry of $M_n$ that is cyclically permuting said cusps. All the remaining cusps of $M_n$ have type $F_1$. Thus $\sigma(M_n) = n$ by Theorem \ref{thm:long-reid}, and 
$\chi(M_n) = n \cdot \chi(M) = n$.

Now, fix $m, n \geq 1$ and let $M_{m,n} \to M_n$ be the cyclic $m$--sheeted covering associated to the restriction $v_{m,n}$ of $p_m \circ v$ to $\ker(h_n)$. Since $v_{m,n}(t_1) = v_{m,n}(t_2) = 0$ and $v_{m,n}(a) = 1$, the subgroup $\pi_1(C) \cap \ker(v_{m,n})$ of $\pi_1(M_{m,n})$ is generated by $t_1$, $t_2$ and $a^m$, and does not contain $a^k$ for any positive $k < m$. If $m$ is odd (or, equivalently, $m \equiv \pm 1 \mod 4$), the associated $n$ cusps of $M_{m,n}$ have type $F_4$. As before, the latter are coherently oriented, and all the remaining cusps are of type $F_1$. Thus $\sigma(M_{m,n}) = n$ by Theorem \ref{thm:long-reid}, and 
$\chi(M_{m,n}) = m \cdot \chi(M_n) = mn$.

The proof of Theorem \ref{thm:main} is thus complete assuming Theorem \ref{thm:construction}. 

\subsection{The construction}

We prove here Theorem \ref{thm:construction}.

The manifold $M$ is specified by the side pairing of a regular ideal $24$--cell, with vertices given in Table \ref{tab1}, by the correspondence between the vertices of paired sides provided in Table \ref{tab2}. We refer the reader to \cite[Section 11.1]{R} for more details on how to build hyperbolic $24$--cell manifolds by using Poincar\'{e}'s fundamental polytope theorem.

\begin{figure}
\centering
\includegraphics[scale=.35]{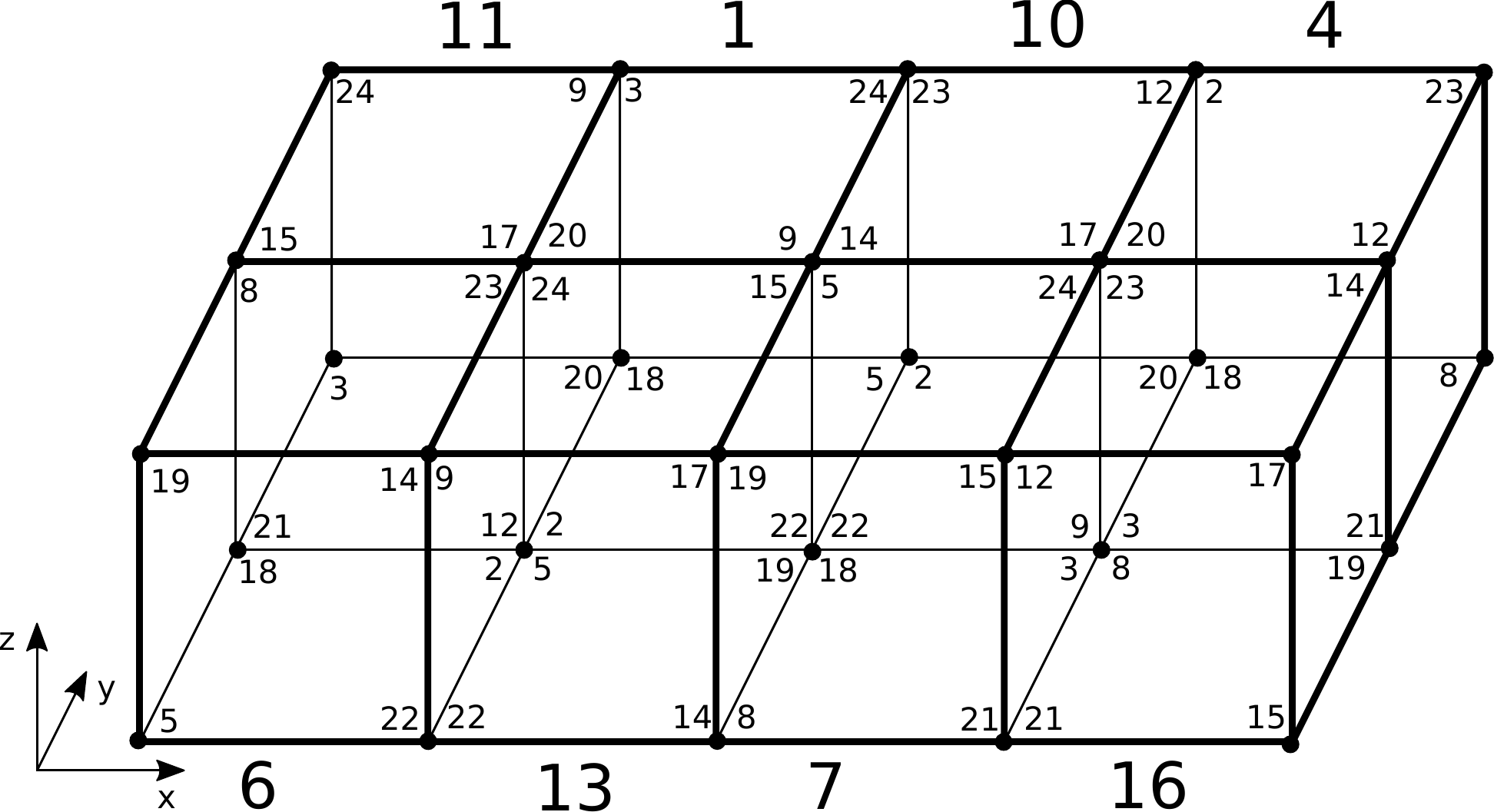}
\caption{\footnotesize A fundamental domain for a horosection of one of the two $3$--torus cusps of $M$ in its universal cover, tessellated by $8$ unit cubes. A cube with label $i$ corresponds to the vertex $v_i$ of the $24$--cell. A vertex of cube $i$ has label $j$ if the corresponding edge of the $24$--cell joins $v_i$ and $v_j$. The resulting Euclidean lattice is generated by the translations along $(4,0,0)$, $(0,2,0)$ and $(0,2,-1)$.} \label{fig:3-torus1}
\end{figure}

\begin{figure}
\centering
\includegraphics[scale=.35]{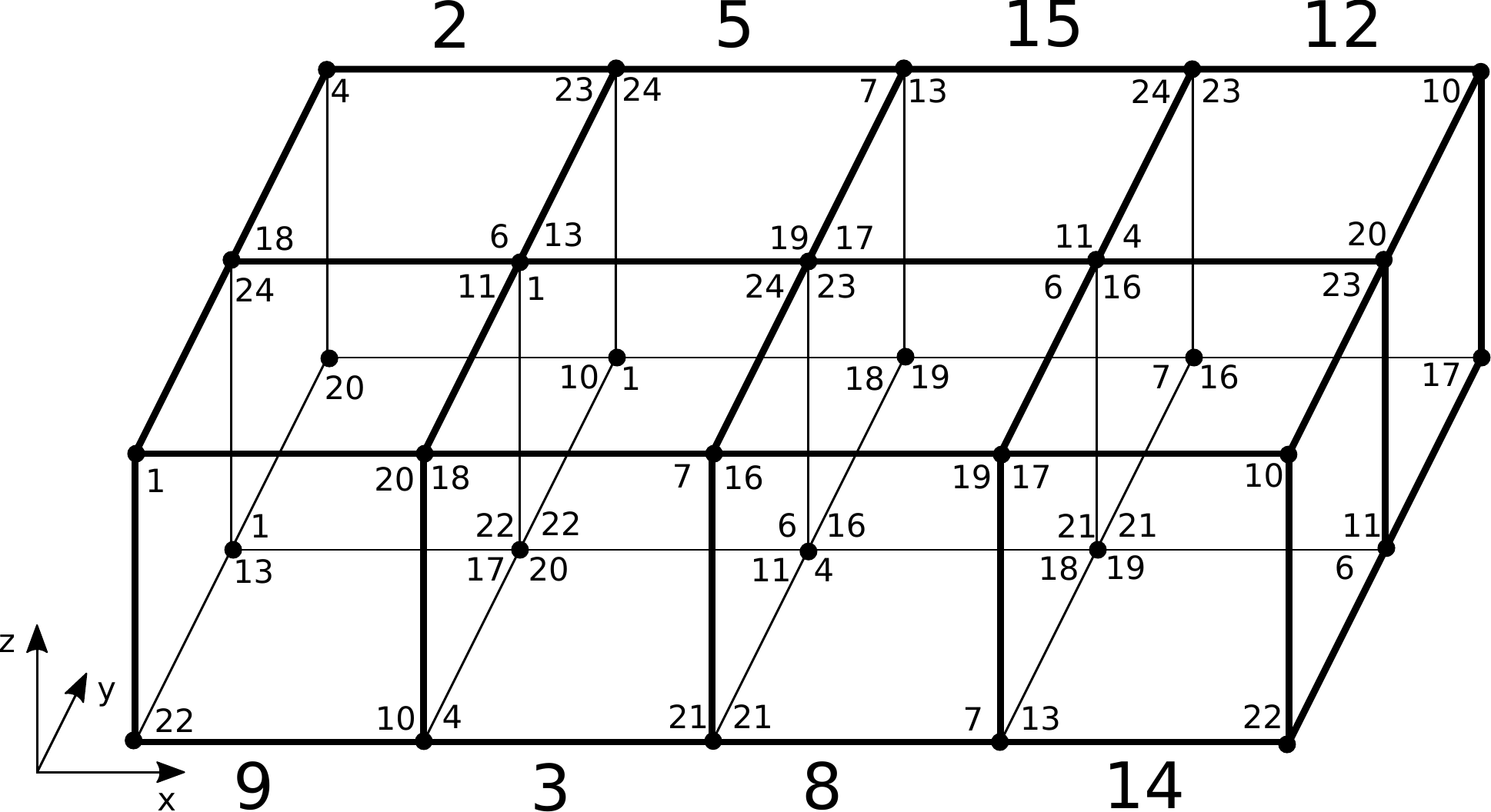}
\caption{\footnotesize A fundamental domain for a horosection of the other $3$--torus cusp of $M$ in its universal cover. The resulting Euclidean lattice is generated by the translations along $(4,0,0)$, $(4,0,-1)$ and $(2,2,0)$.} \label{fig:3-torus2}
\end{figure}

Alternatively, we can give the side pairing by specifying the matrices for each side pairing map.  Since side $1$ is paired to side $2$,
the side pairing map that carries side $2$ to side $1$ is the inverse of that for the map from side $1$ to side $2$, similarly for sides $3$ and $4$, etc.  
The matrices $g_i$ for sides $1, 3, 5, 6, 7, 8, 9, 10, 11, 12, 21$, and $22$
as given in Table \ref{tab3}. These twelve matrices generate the fundamental group of $M$ and the defining relations of this group are determined from the ridge cycles as shown in Table \ref{tab4}.  
There are $96$ ridges in cycles of length $4$, so there are $24$ defining relations. 

The manifold $M$ is orientable with homology groups
\begin{equation*}
H_0(M)=\mathbb{Z},\ H_1(M)=\mathbb{Z}^3,\ H_2(M)=\mathbb{Z}^5,\ H_3(M)=\mathbb{Z}^2,\ H_4(M)=0.    
\end{equation*}

Moreover, $M$ has three cusps: two of type $F_1$, and the other of type $F_4$. In order to facilitate a manual verification, the gluing of the $24$ cubes (the vertex links of the $24$--cell) producing the respective cusp sections are provided in Figures \ref{fig:3-torus1}, \ref{fig:3-torus2} and \ref{fig:quarter-twist}.

\begin{figure}
\centering
\includegraphics[scale=.35]{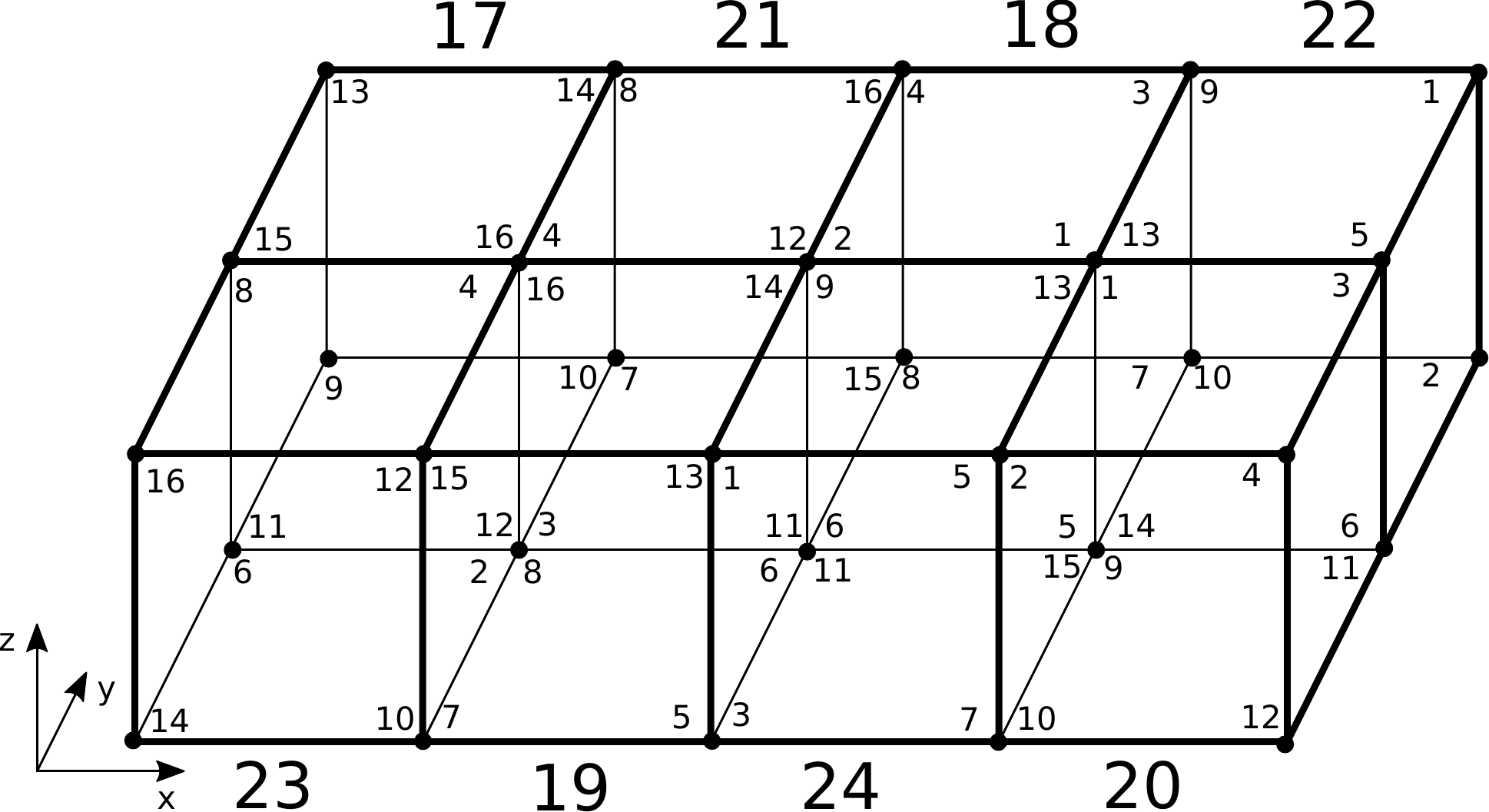}
\caption{\footnotesize A fundamental domain for a horosection of the $F_4$--cusp $C$ of $M$ in its universal cover. The resulting Euclidean lattice is generated by the translations $t_1$ and $t_2$ along $(1,1,0)$ and $(-1,1,0)$, and the rototranslation $a$ along $(0,0,1)$ whose rotational part has vertical axis through the center of cube $21$ (see also Figure \ref{fig:square_torus}).} \label{fig:quarter-twist}
\end{figure}

\begin{figure}
\centering
\includegraphics[scale=.5]{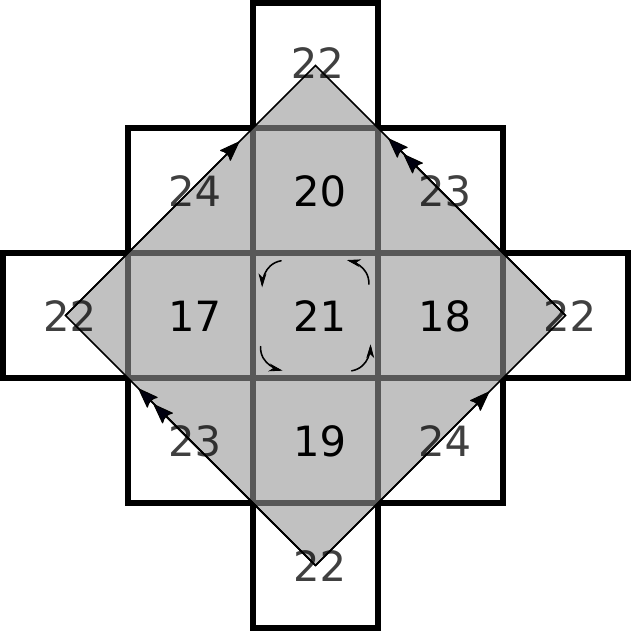}
\caption{\footnotesize A portion of a horizontal slice of the universal cover of a horosection of the $F_4$--cusp $C$ of $M$. The labels agree with the vertex indices of the $24$--cell. The shaded square is a fundamental domain for the translation group $\langle t_1, t_2 \rangle$. The ``quarter-turn'' action of $a$ in the direction orthogonal to the slice is shown by arrows in the center.} \label{fig:square_torus}
\end{figure}

The cusp $C$ of type $F_4$ is the link of the vertex cycle consisting of vertices $v_{17}$, $v_{18}$, $v_{19}$, $v_{20}$, $v_{21}$, $v_{22}$, $v_{23}$, and $v_{24}$. The flat $3$--manifold for the link of $C$ is realised as a gluing of $8$ cubes, giving a presentation for $\pi_1(C)$ with $24$ generators having $7$ singleton and $24$ length four defining relations (such a presentation can certainly be simplified considerably, although we do not need to do this). The $8$ cubes glue together in a pattern corresponding to a square $2$--torus tiled by $8$ squares with sides at $\frac{\pi}{4}$ angle to the axes of the torus as in Figure \ref{fig:square_torus}, multiplied by an interval in order to produce $8$ cubes. The top and bottom tori are identified with a $\frac{\pi}{2}$ twist of the square torus, which is exactly the gluing pattern for the $F_4$ manifold.

Extracting the words in the generators of $\pi_1(M)$ giving parabolic elements stabilising vertex $v_{21}$ produces the generators for $\pi_1(C)$, see the accompanying Tables \ref{tab5} and \ref{tab6}. The two homomorphisms $h$ and $v$ are defined in Table \ref{tab7}.

The proof of Theorem \ref{thm:construction} is now complete.

\begin{table}[h]
$$
\begin{array}{ll}
 v_{1} = \left(-\frac{1}{2},-\frac{1}{2},-\frac{1}{2},-\frac{1}{2},1\right) &
 v_{2} = \left(-\frac{1}{2},-\frac{1}{2},-\frac{1}{2},\frac{1}{2},1\right) \\[5pt]
 v_{3} = \left(-\frac{1}{2},-\frac{1}{2},\frac{1}{2},-\frac{1}{2},1\right) &
 v_{4} = \left(-\frac{1}{2},-\frac{1}{2},\frac{1}{2},\frac{1}{2},1\right) \\[5pt]
 v_{5} = \left(-\frac{1}{2},\frac{1}{2},-\frac{1}{2},-\frac{1}{2},1\right) &
 v_{6} = \left(-\frac{1}{2},\frac{1}{2},-\frac{1}{2},\frac{1}{2},1\right) \\[5pt]
 v_{7} = \left(-\frac{1}{2},\frac{1}{2},\frac{1}{2},-\frac{1}{2},1\right) &
 v_{8} = \left(-\frac{1}{2},\frac{1}{2},\frac{1}{2},\frac{1}{2},1\right) \\[5pt]
 v_{9} = \left(\frac{1}{2},-\frac{1}{2},-\frac{1}{2},-\frac{1}{2},1\right) &
 v_{10} = \left(\frac{1}{2},-\frac{1}{2},-\frac{1}{2},\frac{1}{2},1\right) \\[5pt]
 v_{11} = \left(\frac{1}{2},-\frac{1}{2},\frac{1}{2},-\frac{1}{2},1\right) &
 v_{12} = \left(\frac{1}{2},-\frac{1}{2},\frac{1}{2},\frac{1}{2},1\right) \\[5pt]
 v_{13} = \left(\frac{1}{2},\frac{1}{2},-\frac{1}{2},-\frac{1}{2},1\right) &
 v_{14} = \left(\frac{1}{2},\frac{1}{2},-\frac{1}{2},\frac{1}{2},1\right) \\[5pt]
 v_{15} = \left(\frac{1}{2},\frac{1}{2},\frac{1}{2},-\frac{1}{2},1\right) &
 v_{16} = \left(\frac{1}{2},\frac{1}{2},\frac{1}{2},\frac{1}{2},1\right) \\[5pt]
 v_{17} = (1,0,0,0,1) &
 v_{18} = (-1,0,0,0,1) \\[5pt]
 v_{19} = (0,1,0,0,1) &
 v_{20} = (0,-1,0,0,1) \\[5pt]
 v_{21} = (0,0,1,0,1) &
 v_{22} = (0,0,-1,0,1) \\[5pt]
 v_{23} = (0,0,0,1,1) &
 v_{24} = (0,0,0,-1,1) \\[5pt]
\end{array}
$$
\caption{\footnotesize The vertices of an ideal regular $24$--cell in the hyperboloid model of hyperbolic $4$--space.}\label{tab1}
\end{table}

\begin{table}[ht]
\vspace{1.75in}
$$
\begin{array}{ccll}
 \vspace{.3cm} 
 \text{From side} & \text{To side} & \text{Vertex map}\\
 1 & 2 & (13,14,15,16,17,19)&\mapsto\ \  (7,8,5,6,19,18) \\
 2 & 1 & (5,6,7,8,18,19)&\mapsto\ \  (15,16,13,14,19,17) \\
 3 & 4 & (9,10,11,12,17,20)&\mapsto\ \  (3,4,1,2,20,18) \\
 4 & 3 & (1,2,3,4,18,20)&\mapsto\ \  (11,12,9,10,20,17) \\
 5 & 14 & (11,12,15,16,17,21)&\mapsto\ \  (6,2,8,4,23,18) \\
 6 & 13 & (3,4,7,8,18,21)&\mapsto\ \  (12,16,10,14,23,17) \\
 7 & 16 & (9,10,13,14,17,22)&\mapsto\ \  (3,7,1,5,24,18) \\
 8 & 15 & (1,2,5,6,18,22)&\mapsto\ \  (13,9,15,11,24,17) \\
 9 & 18 & (7,8,15,16,19,21)&\mapsto\ \  (10,2,12,4,23,20) \\
 10 & 17 & (3,4,11,12,20,21)&\mapsto\ \  (8,16,6,14,23,19) \\
 11 & 20 & (5,6,13,14,19,22)&\mapsto\ \  (3,11,1,9,24,20) \\
 12 & 19 & (1,2,9,10,20,22)&\mapsto\ \  (13,5,15,7,24,19) \\
 13 & 6 & (10,12,14,16,17,23)&\mapsto\ \  (7,3,8,4,21,18) \\
 14 & 5 & (2,4,6,8,18,23)&\mapsto\ \  (12,16,11,15,21,17) \\
 15 & 8 & (9,11,13,15,17,24)&\mapsto\ \  (2,6,1,5,22,18) \\
 16 & 7 & (1,3,5,7,18,24)&\mapsto\ \  (13,9,14,10,22,17) \\
 17 & 10 & (6,8,14,16,19,23)&\mapsto\ \  (11,3,12,4,21,20) \\
 18 & 9 & (2,4,10,12,20,23)&\mapsto\ \  (8,16,7,15,21,19) \\
 19 & 12 & (5,7,13,15,19,24)&\mapsto\ \  (2,10,1,9,22,20) \\
 20 & 11 & (1,3,9,11,20,24)&\mapsto\ \  (13,5,14,6,22,19) \\
 21 & 23 & (4,8,12,16,21,23)&\mapsto\ \  (11,3,15,7,21,24) \\
 22 & 24 & (2,6,10,14,22,23)&\mapsto\ \  (5,13,1,9,22,24) \\
 23 & 21 & (3,7,11,15,21,24)&\mapsto\ \  (8,16,4,12,21,23) \\
 24 & 22 & (1,5,9,13,22,24)&\mapsto\ \  (10,2,14,6,22,23) \\[5pt]
\end{array}
$$
\caption{\footnotesize Side pairings defining $M$.}\label{tab2}
\end{table} 

\begin{table}
$$
\begin{array}{ll}
 g_1 = \left(
\begin{array}{ccccc}
 2 & 1 & 0 & 0 & -2 \\
 -1 & -2 & 0 & 0 & 2 \\
 0 & 0 & -1 & 0 & 0 \\
 0 & 0 & 0 & 1 & 0 \\
 -2 & -2 & 0 & 0 & 3 \\
\end{array}
\right) &
 g_3 = \left(
\begin{array}{ccccc}
 2 & -1 & 0 & 0 & -2 \\
 1 & -2 & 0 & 0 & -2 \\
 0 & 0 & -1 & 0 & 0 \\
 0 & 0 & 0 & 1 & 0 \\
 -2 & 2 & 0 & 0 & 3 \\
\end{array}
\right) \\[30pt]
 g_5 = \left(
\begin{array}{ccccc}
 2 & 0 & 1 & 0 & -2 \\
 0 & 0 & 0 & -1 & 0 \\
 0 & 1 & 0 & 0 & 0 \\
 -1 & 0 & -2 & 0 & 2 \\
 -2 & 0 & -2 & 0 & 3 \\
\end{array}
\right) &
 g_6 = \left(
\begin{array}{ccccc}
 2 & 0 & -1 & 0 & 2 \\
 0 & 0 & 0 & 1 & 0 \\
 0 & -1 & 0 & 0 & 0 \\
 1 & 0 & -2 & 0 & 2 \\
 2 & 0 & -2 & 0 & 3 \\
\end{array}
\right) \\[30pt]
 g_7 = \left(
\begin{array}{ccccc}
 2 & 0 & -1 & 0 & -2 \\
 0 & 0 & 0 & 1 & 0 \\
 0 & -1 & 0 & 0 & 0 \\
 1 & 0 & -2 & 0 & -2 \\
 -2 & 0 & 2 & 0 & 3 \\
\end{array}
\right) &
 g_8 = \left(
\begin{array}{ccccc}
 2 & 0 & 1 & 0 & 2 \\
 0 & 0 & 0 & -1 & 0 \\
 0 & 1 & 0 & 0 & 0 \\
 -1 & 0 & -2 & 0 & -2 \\
 2 & 0 & 2 & 0 & 3 \\
\end{array}
\right) \\[30pt]
 g_9 = \left(
\begin{array}{ccccc}
 0 & 0 & 0 & -1 & 0 \\
 0 & 2 & 1 & 0 & -2 \\
 1 & 0 & 0 & 0 & 0 \\
 0 & -1 & -2 & 0 & 2 \\
 0 & -2 & -2 & 0 & 3 \\
\end{array}
\right) &
 g_{10} = \left(
\begin{array}{ccccc}
 0 & 0 & 0 & 1 & 0 \\
 0 & 2 & -1 & 0 & 2 \\
 -1 & 0 & 0 & 0 & 0 \\
 0 & 1 & -2 & 0 & 2 \\
 0 & 2 & -2 & 0 & 3 \\
\end{array}
\right) \\[30pt]
 g_{11} = \left(
\begin{array}{ccccc}
 0 & 0 & 0 & 1 & 0 \\
 0 & 2 & -1 & 0 & -2 \\
 -1 & 0 & 0 & 0 & 0 \\
 0 & 1 & -2 & 0 & -2 \\
 0 & -2 & 2 & 0 & 3 \\
\end{array}
\right) &
 g_{12} = \left(
\begin{array}{ccccc}
 0 & 0 & 0 & -1 & 0 \\
 0 & 2 & 1 & 0 & 2 \\
 1 & 0 & 0 & 0 & 0 \\
 0 & -1 & -2 & 0 & -2 \\
 0 & 2 & 2 & 0 & 3 \\
\end{array}
\right) \\[30pt]
 g_{21} = \left(
\begin{array}{ccccc}
 0 & -1 & 0 & 0 & 0 \\
 1 & 0 & 0 & 0 & 0 \\
 0 & 0 & -1 & -2 & 2 \\
 0 & 0 & 2 & 1 & -2 \\
 0 & 0 & -2 & -2 & 3 \\
\end{array}
\right) &
 g_{22} = \left(
\begin{array}{ccccc}
 0 & 1 & 0 & 0 & 0 \\
 -1 & 0 & 0 & 0 & 0 \\
 0 & 0 & -1 & 2 & -2 \\
 0 & 0 & -2 & 1 & -2 \\
 0 & 0 & 2 & -2 & 3 \\
\end{array}
\right) \\[30pt]
\end{array}
$$
\caption{\footnotesize Generators of $\pi_1(M)$ in $\SO(4,1)$.}\label{tab3}
\end{table}

\begin{table}
$$
\begin{array}{llll}
 g_3g_{10}^{-1}g_{22}^{-1}g_8 &
 g_3g_5^{-1}g_{22}^{-1}g_{12} &
 g_7g_8g_{12}^{-1}g_{11}^{-1} &
 g_3g_{11}g_{22}^{-1}g_7^{-1} \\[5pt]
 g_3g_8g_{22}^{-1}g_{11}^{-1} &
 g_7g_{11}^{-1}g_{12}^{-1}g_8 &
 g_1g_8g_{22}g_{12}^{-1} &
 g_1g_7^{-1}g_{22}g_9 \\[5pt]
 g_1g_{12}g_{22}g_7^{-1} &
 g_1g_{11}^{-1}g_{22}g_6 &
 g_3g_9g_{21}g_5^{-1} &
 g_3g_6g_{21}g_9^{-1} \\[5pt]
 g_5g_6g_{10}^{-1}g_9^{-1} &
 g_5g_{10}^{-1}g_{11}^{-1}g_8 &
 g_6g_9^{-1}g_{12}^{-1}g_7 &
 g_3g_{12}^{-1}g_{21}g_6 \\[5pt]
 g_3g_7^{-1}g_{21}g_{10} &
 g_6g_7g_{12}^{-1}g_9^{-1} &
 g_5g_8g_{11}^{-1}g_{10}^{-1} &
 g_1g_6g_{21}^{-1}g_{10}^{-1} \\[5pt]
 g_1g_5^{-1}g_{21}^{-1}g_{11} &
 g_5g_9^{-1}g_{10}^{-1}g_6 &
 g_1g_{10}g_{21}^{-1}g_5^{-1} &
 g_1g_9^{-1}g_{21}^{-1}g_8 \\[5pt]
\end{array}
$$
\caption{\footnotesize Defining relations for $\pi_1(M)$.}\label{tab4}
\end{table}

\begin{table}
$$
\begin{array}{ll}
 t_1 = g_5^{-1}g_8^{-1}g_{12}g_9= & \left(
\begin{array}{ccccc}
 1 & 0 & -4 & 0 & 4 \\
 0 & 1 & 4 & 0 & -4 \\
 4 & -4 & -15 & 0 & 16 \\
 0 & 0 & 0 & 1 & 0 \\
 4 & -4 & -16 & 0 & 17 \\
\end{array}
\right) \\
 t_2 = g_9^{-1}g_{10}^{-1}g_5g_6= & \left(
\begin{array}{ccccc}
 1 & 0 & -4 & 0 & 4 \\
 0 & 1 & -4 & 0 & 4 \\
 4 & 4 & -15 & 0 & 16 \\
 0 & 0 & 0 & 1 & 0 \\
 4 & 4 & -16 & 0 & 17 \\
\end{array}
\right) \\
 t_3 = g_{21}^4= & \left(
\begin{array}{ccccc}
 1 & 0 & 0 & 0 & 0 \\
 0 & 1 & 0 & 0 & 0 \\
 0 & 0 & -31 & -8 & 32 \\
 0 & 0 & 8 & 1 & -8 \\
 0 & 0 & -32 & -8 & 33 \\
\end{array}
\right) \\
 a = g_{21}= & \left(
\begin{array}{ccccc}
 0 & -1 & 0 & 0 & 0 \\
 1 & 0 & 0 & 0 & 0 \\
 0 & 0 & -1 & -2 & 2 \\
 0 & 0 & 2 & 1 & -2 \\
 0 & 0 & -2 & -2 & 3 \\
\end{array}
\right) \\[5pt]
\end{array}
$$
\caption{\footnotesize Generators of $\pi_1(C)$ in $\SO(4,1)$.}\label{tab5}
\end{table}

\vspace{4cm}

\begin{table}
$$
\begin{array}{llll}
 a^4t_3^{-1} &
 at_1a^{-1}t_2^{-1} &
 at_2a^{-1}t_1 &
 at_3a^{-1}t_3^{-1} \\[5pt]
 t_1t_2t_1^{-1}t_2^{-1} &
 t_1t_3t_1^{-1}t_3^{-1} &
 t_2t_3t_2^{-1}t_3^{-1} &\\[5pt]
\end{array}
$$
\caption{\footnotesize Defining relations for $\pi_1(C)$.}\label{tab6}
\end{table}

\vspace{4cm}

\begin{table}
$$
\begin{array}{lllllllll}
h(g_1) = 1\ & h(g_3) = 1\ & h(g_5) = 2\ & h(g_6) = 0 \\[5pt]
h(g_7) = 2\ & h(g_8) = 0\ & h(g_9) = 1\ & h(g_{10}) = 1 \\[5pt]
h(g_{11}) = 1\ & h(g_{12}) = 1\ & h(g_{21}) = 0\ & h(g_{22}) = 0 \\[10pt]
v(g_1) = 2\ & v(g_3) = 0\ & v(g_5) = 2\ & v(g_6) = 0 \\[5pt]
v(g_7) = 2\ & v(g_8) = 0\ & v(g_9) = 1\ & v(g_{10}) = 1 \\[5pt]
v(g_{11}) = 1\ & v(g_{12}) = 1\ & v(g_{21}) = 1\ & v(g_{22}) = -1 \\[5pt]
\end{array}
$$
\caption{\footnotesize The two homomorphisms $h, v \colon \pi_1(M) \to \matZ$.}\label{tab7}
\end{table} 

\FloatBarrier

\end{document}